\documentclass[reqno,draft]{amsart}
\usepackage{amsfonts}%
\usepackage{amssymb}
\usepackage{mathrsfs}
\usepackage{appendix}
\usepackage{color}
\usepackage{verbatim}
\usepackage[marginal]{footmisc}

\allowdisplaybreaks

 \numberwithin{equation}{section}

\newtheorem{Theorem}{Theorem}[section]
\newtheorem{Lemma}{Lemma}[section]
\newtheorem{Proposition}{Proposition}[section]
\newtheorem{Def}{Definition}[section]
\newtheorem{Remark}{Remark}[section]
\newtheorem{Corollary}{Corollary}[section]

\begin{document}
\title[Asymptotics of Dirichlet Problems to Fractional $p$-Laplacian]
 {Asymptotics of Dirichlet Problems to Fractional $p$-Laplacian Functionals: Approach in De Giorgi sense}

\author{Raphael Feng LI}
\address{ Department of Mathematics,
 Nagoya University, Nagoya 464-8602, Japan}
\email{d15002m@math.nagoya-u.ac.jp}

\begin{abstract}
In this paper we firstly study the limit of minimizers of the fractional $W^{s,p}$-norms as $p\rightarrow+\infty$ in De Giorgi sense. In particular, we analyzed the $\Gamma$-convergence of non-homogeneous Dirichlet boundary problem for fractional $p$-Laplacian in this approximation process, and proved that as $p\rightarrow+\infty$ the minimizers of fractional $p$-Laplacian with Dirichlet boundary $\Gamma$-converges to a minimizer of H\"{o}lder $\infty$-Laplacian under the same Dirichlet boundary condition.

On the other hand, we first investigate the asymptotic behaviour of non-homogeneous fractional $p$-functionals when $k\rightarrow s$ from above; then we study the approximation process as $k\rightarrow s$ from below of a free fractional $p$-functional, during which we will find some special phenomenon different from the case from above. Both of the way to dispose these two asymptotic directions are in the De Giorgi sense.
\end{abstract}


\maketitle


\section{Organization of This Paper}
In this paper, we mainly use the $\Gamma$-convergence introduced by E. De Giorgi in 1970's to investigate some approximation phenomenons on the fractional $p$-Laplacian equations and related functionals. The $\Gamma$-convergence is defined as:
\begin{Def}[$\Gamma$-convergence]\label{GammaConvDef}
Let $X$ be a metric space. A sequence $\{E_n\}$ of functionals $E_n:X\rightarrow\overline{\mathbb{R}}:=\mathbb{R}\cup\{\infty\}$ is said to $\Gamma(X)$-convergence to $E_{\infty}:X\rightarrow\overline{\mathbb{R}}$, and we write $\Gamma(X)$-$\lim\limits_{n\rightarrow+\infty}E_n=E_{\infty}$, if the following hold:\\
(i) for every $u\in X$ and $\{u_n\}\subset X$ such that $u_n\rightarrow u$ in $X$, we have
$$
E_{\infty}(u)\leq\liminf\limits_{n\rightarrow+\infty}E_n(u_n);
$$
(ii) for every $u\in X$ there exists a sequence $\{u_n\}\subset X$(called a recovery sequence) such that $u_n\rightarrow u$ in $X$
and
$$
E_{\infty}(u)\geq\limsup\limits_{n\rightarrow+\infty}E_n(u_n).
$$
\end{Def}
For further information, one can refer to \cite{Braides,DalMaso}.
\medskip

In section \ref{ptoinfty}, we utilize the settings supposed as, $0<\alpha<1$, $p>\frac{N}{\alpha}$ and $\Omega$ being an open bounded domain in $\mathbb{R}^N$ with Lipschitz boundary. We mainly study the approximation process of the minimizers of fractional $p$-functionals
$$
\int_{\Omega\times\Omega}\frac{|u(x)-u(y)|^p}{|x-y|^{\alpha p}}dxdy,
$$
as $p\rightarrow+\infty$ under suitable non-homogeneous Dirichlet condition in some admissible space $X$.

In section \ref{svaryingprocess}, we investigate the behaviour of homogeneous Dirichlet problem of the fractional $p$-functional
$$
\min_{u\in Y}\left(
\int_{B_{tR}(\Omega)\times B_{tR}(\Omega)}\frac{|u(x)-u(y)|^p}{|x-y|^{N+kp}}dxdy+\int_{\Omega}fudx
\right),
$$
when $\{k\}\subset(0,1)$ decreases to some $s\in(0,1)$. Here $B_{tR}(\Omega)$ is define as the $N$-dimensional ball with diameter $tR$ located at the same center as the smallest ball containing $\Omega$, in which, $t>1$ and $R$ is the diameter of $\Omega$. We assume that $0<s<k<1$, $p\in(1,+\infty)$ and $\Omega$ being an open bounded set in $\mathbb{R}^N$ without regularity assumption on $\partial\Omega$. In order to investigate the asymptotics smoothly, we introduce a $relative$-$nonlocal$ Sobolev space $\widetilde{W}^{s,p}_{0,tR}(\Omega)$. And then in some admissible space $Y$, we investigate the asymptotic behaviours of the functionals

Then under the case $k\rightarrow s$ from below, we assume that $0<k<s<1$, $p\in(1,+\infty)$ and $\Omega$ being an open bounded set in $\mathbb{R}^N$, without further regularity assumption on $\partial\Omega$. Then inspired by \cite{DM} (see also \cite{Li}), we construct a space
$$
W^{s^-,p}_0(\Omega):=\bigcap\limits_{0<k<s}(W^{k,p}_0(\Omega)\cap W^{s,p}(\Omega)),
$$
to study the convergence of the free functional
$$
\int_{\Omega\times\Omega}\frac{|u(x)-u(y)|^p}{|x-y|^{N+kp}}dxdy
$$
when $k$ increases to some $s\in(0,1)$. We will see that we can not get a ideal result as in the case approximating from above. And as a byproduct, we give an equivalence description between the spaces $W^{s^-,p}_0(\Omega)$ and $W^{s,p}_0(\Omega)$ in De Giorgi sense. For more information on this topic, one can see \cite{DM,Li}.

\section{Asymptotic Behaviour as $p\rightarrow+\infty$}\label{ptoinfty}

\subsection{Setting of the Problem}

Let $\Omega$ be a bounded domain in $\mathbb{R}^N$ with Lipschitz boundary. It's well-known that the minimizers $u_p$ of the integrals
$$\int_{\Omega}|\nabla u|^p,$$
under suitable conditions, as p$\rightarrow+\infty$, approximate to the minimizer $u$ of the equation
$$\Delta_{\infty}u=\sum_{i,j=1,2,...,N}u_{ij}u_iu_j=0\ on\ \Omega,$$
with $u_i=\frac{\partial u}{\partial x_i}$ and $u_{ij}=\frac{\partial^2u}{\partial x_i \partial x_j}$, which is usually referred to as $\infty$-Laplacian equation, introduced by Aronsson in the fundamental work \cite{Aronsson3, Aronsson4} as the Euler-Lagrange equations associated to the functional $$\|\nabla u\|_{L^{\infty}(\Omega)}.$$ Here, the weak solution $u$ to the $\infty$-Laplacian equation is understood in the viscosity sense. One can refer to, for instance, \cite{Aronsson1, Aronsson2, Aronsson3, Aronsson4, Aronsson5, ACJ, BDBM, BJW1, BJW2} for the limitation discussion as $p\rightarrow+\infty$. Moreover, $u$ is known as a local minimizer up to a Lipschitz extension, for which, one can refer to \cite{ACJ}. One can notice that the approximation process above is pointwise, and in \cite{BN,BM} one can find another approximation approach for variable $p(\cdot)$ based on $\Gamma$-convergence, which is also our concentration in this paper.

In this paper, we are concerned with the fractional case.

We study the Dirichlet problem and the minimizers of the functional
\begin{align}\label{FracFunc}
\int_{\Omega\times\Omega}\frac{|u(x)-u(y)|^p}{|x-y|^{\alpha p}}dxdy,
\end{align}
for $p\alpha>N$ ($N$ is the dimension of $\mathbb{R}^N$) with $\alpha\in(0,1)$, and $\Omega$ being a bounded domain in $\mathbb{R}^N$. For the fractional Sobolev semi-norm $W^{s,p}(\Omega)$ ($s\in(0,1)$) defined as
$$
[u]^p_{W^{s,p}(\Omega)}:=\int_{\Omega\times\Omega}\frac{|u(x)-u(y)|^p}{|x-y|^{N+sp}}dxdy,
$$
in the limit case as $p\rightarrow+\infty$, the fractional functional approximates to, formally,
\begin{align}\label{InftyFunc}
\|\frac{|u(x)-u(y)|}{|x-y|^{s}}\|_{L^{\infty}(\Omega\times\Omega)}.
\end{align}

In general, the Euler-Lagrange equations of the fractional functional (\ref{FracFunc}) is
\begin{align}\label{FracEL}
\mathcal{L}^{\alpha}_pu(x):=\int_{\Omega}|\frac{u(x)-u(y)}{|x-y|^{\alpha}}|^{p-1}\frac{sgn(u(x)-u(y))}{|x-y|^{\alpha}}dy=0\ in\ \Omega.
\end{align}

In viscosity sense, as $p\rightarrow+\infty$, the equation (\ref{FracEL}) should converge to the H$\ddot{o}$lder $\infty$-Laplacian equation (refer to \cite{CLM}), defined as
\begin{align}\label{InftyEL}
L^{\alpha}u=0\ in\ \Omega,
\end{align}
with the definition of operator $L^{\alpha}$
\begin{align}\label{HILapOpe}
L^{\alpha}u(x):=\sup\limits_{y\in\overline{\Omega},y\neq x}\frac{u(y)-u(x)}{|y-x|^{\alpha}}+\inf\limits_{y\in\overline{\Omega},y\neq x}
\frac{u(y)-u(x)}{|y-x|^{\alpha}}\ for\ x\in\Omega.
\end{align}
For the research on the Dirichlet problem of Euler-Lagrange equations of functional (\ref{FracFunc})
\begin{equation}\label{FracLapEqu}
\left\{
\begin{array}{lr}
\mathcal{L}^{\alpha}_pu(x)=f(x)&\ in\ \Omega,\\
u=g&\ on\ \partial\Omega.
\end{array}
\right.
\end{equation}
one can refer to \cite{CLM, Lindgren, LL}. We also want to mention that the boundary condition can be changed to the fully nonlocal case, that is $u=g\ on\ \mathbb{R}^N\setminus\Omega$, and then we would work on the space $\widetilde{W}^{s,p}_0(\Omega)$ defined as the complete closure of $C^{\infty}_0(\Omega)$ under the norm $W^{s,p}(\mathbb{R}^N)$. For the research in this direction, one can refer to \cite{DCKP, FP, Palatucci}, and a final generalization comment in \cite{CLM}.

For the Dirichlet problem of H$\ddot{o}$lder $\infty$-Laplacian equations, we denote
\begin{equation}\label{HILapEqu}
\left\{
\begin{aligned}
L^{\alpha}u=f\ in \ \Omega,\\
u=g\ on\ \partial\Omega.
\end{aligned}
\right.
\end{equation}
In \cite{CLM} one can see that under suitable conditions when $p\rightarrow+\infty$ is large enough, the weak solutions of Dirichlet problem of (\ref{FracLapEqu}) converge to the weak solutions of the equations (\ref{HILapEqu}) in the viscosity sense.
For the readers' convenience, we list the results below without proof.
\begin{Theorem}[\cite{CLM} Theorem 1.1, limit equation as $p\rightarrow+\infty$]
Let $\alpha\in(0,1]$ and if $\alpha = 1$ assume $N\geq2$. Consider a bounded Lipschitz domain $\Omega$ in $\mathbb{R}^N$, and
boundary data $g\in C^{0,\alpha}(\partial\Omega)$. For any $p>2N/\alpha$, there exists a unique minimizer $u_p$ of (\ref{FracFunc}) satisfying $u=g$ on $\partial\Omega$. Moreover, as $p\rightarrow+\infty$, we have $u_p\rightarrow u_{\infty}$
uniformly in $\overline{\Omega}$ and $u_{\infty}\in C^{0,\alpha}(\overline{\Omega})$ is a viscosity solution of (\ref{InftyEL}).
\end{Theorem}

One can see that under suitable conditions the minimizers $u$ exhibit $\alpha$-H$\ddot{o}$lder continuity up to the boundary. So it is safe to assume that the boundary value $g|_{\partial\Omega}$ is $\alpha$-H$\ddot{o}$lder continuous when $p$ is large enough.

The first half of this paper is to investigate the convergence of fractional functional (\ref{FracFunc}) to the infinity functional (\ref{InftyFunc}) when $p\rightarrow+\infty$ in De Giorgi sense. Then based on this, we also investigate the compatibility of non-homogeneous Dirichlet problems during the process $p\rightarrow+\infty$ of the functional (\ref{FracFunc}).

We want seize the chance to mention the following implicit representation of viscosity solution to (\ref{HILapEqu}) when $f=0$. We just give the statement of the theorem, and for the proof details one can refer to \cite{CLM}.
\begin{Theorem}[\cite{CLM} Theorem 1.5, existence for general $\it{f}$, partial uniqueness]\label{ParUniqueness}
Let $\alpha\in(0,1]$, $\Omega$ be a bounded open domain, $g\in C(\partial\Omega)$ and $f\in C(\Omega)\cap L^{\infty}(\Omega)$.
\begin{itemize} \itemsep -2pt
    \item (Existence) Then there exists a viscosity solution $u\in C(\overline{\Omega})$ of (\ref{HILapEqu}).
    \item (Partial uniqueness) Assume $f=0$. Then the viscosity solution $u\in C(\overline{\Omega})$ of (\ref{HILapEqu}) is unique and is defined implicitly by the following:
    \begin{equation}\label{SolFormula}
    u(x)=\left\{
    \begin{array}{lr}
    g(x)& if\ x\in\partial\Omega,\\
    a\ with\ \ell_x(a)=0& if\ x\in\Omega,
    \end{array}
    \right.
    \end{equation}
    where
    $$
    \ell_x(a)=\sup\limits_{y\in\partial\Omega}\frac{g(y)-a}{|y-x|^{\alpha}}+
    \inf\limits_{y\in\partial\Omega}\frac{g(y)-a}{|y-x|^{\alpha}}.
    $$
\end{itemize}
\end{Theorem}

\subsection{Main Results}

In order to neatly present the subject, we first need some definitions. The natural setting for variational functional of the operator $\mathcal{L}^{\alpha}_p$ in the domain $\Omega\subseteq\mathbb{R}^N$ is the space $W^{s,p}_0(\Omega)$ with $s=\alpha-N/p$, $\alpha\in(0,1)$ and $p\alpha>N$, defined as the completion of $C^{\infty}_0(\Omega)$ with respect to the standard Gagliardo semi-norm
$$
[u]_{W^{s,p}(\Omega)}:=\left(\int_{\Omega\times\Omega}\frac{|u(x)-u(y)|^p}{|x-y|^{N+sp}}dxdy\right)^{\frac{1}{p}};
$$
if $p=\infty$, the semi-norms $W^{s,\infty}(\overline\Omega)$ and $W^{s,\infty}(\Omega)$ are respectively defined by
$$
[u]_{W^{s,\infty}(\overline\Omega)}:=\sup\limits_{x\neq y,x,y\in\overline\Omega}|\frac{u(x)-u(y)}{|x-y|^{s}}|,
$$
and
$$
[u]_{W^{s,\infty}(\Omega)}:=\sup\limits_{x\neq y,x,y\in\Omega}|\frac{u(x)-u(y)}{|x-y|^{s}}|.
$$

In all that follows, for $\alpha\in(0,1)$ and $q\alpha>N$, we define $E_{\alpha,p}:L^q(\Omega)\rightarrow[0,\infty]$ by
$$
E_{\alpha,p}(u)=
\left\{
\begin{array}{lr}
\left(\int_{\Omega\times\Omega}\frac{|u(x)-u(y)|^p}{|x-y|^{\alpha p}}dxdy\right)^{\frac{1}{p}} & if\ u\in W^{s,p}(\Omega)\ (s=\alpha-N/q),\\
\infty & otherwise.
\end{array}
\right.
$$
Define $\overline{E}_{\alpha,\infty}:L^q(\Omega)\rightarrow[0,\infty]$ by
$$
\overline{E}_{\alpha,\infty}(u)=
\left\{
\begin{array}{lr}
\sup\limits_{x\neq y,x,y\in\overline\Omega}|\frac{u(x)-u(y)}{|x-y|^{\alpha}}| & if\ u\in W^{s,\infty}(\overline\Omega)\ (s=\alpha),\\
\infty & otherwise;
\end{array}
\right.
$$
and $E_{\alpha,\infty}:L^q(\Omega)\rightarrow[0,\infty]$ by
$$
E_{\alpha,\infty}(u)=
\left\{
\begin{array}{lr}
\sup\limits_{x\neq y,x,y\in\Omega}|\frac{u(x)-u(y)}{|x-y|^{\alpha}}| & if\ u\in W^{s,\infty}(\Omega)\ (s=\alpha),\\
\infty & otherwise.
\end{array}
\right.
$$

The first result concerns the $\Gamma(L^q(\Omega))$-convergence of the functional
$$
E_{\alpha,q}
$$
to the $\alpha$-infinity functional
$$
E_{\alpha,\infty}\ \ and\ \ \overline{E}_{\alpha,\infty}
$$
respectively, as $u_p\rightarrow u$ in $L^q(\Omega)$ strongly for different suitable $q>\frac{N}{\alpha}$.

\begin{Theorem}[Asymptotic behaviour of $p\rightarrow+\infty$]\label{GammaConvinfty}
Let $\alpha\in(0,1)$ and $\Omega$ be a bounded Lipschitz domain in $\mathbb{R}^N$. We consider $\{p\}_p$ as a strictly increasing sequence going to $+\infty$. Then we have

(i) $\Gamma(L^q(\Omega))$-$\lim\limits_{p\rightarrow+\infty}E_{\alpha,p}=E_{\alpha,\infty}$ with some $q>\frac{N}{\alpha}$;

(ii) $\Gamma(L^q(\Omega))$-$\lim\limits_{p\rightarrow+\infty}E_{\alpha,p}=\overline{E}_{\alpha,\infty}$ with some $q>\frac{2N}{\alpha}$.
\end{Theorem}

The proof of this theorem follows from Proposition \ref{limsup} and Proposition \ref{liminf} below.

\begin{Remark}
The reason why we utilize here $\Omega$ being a domain, not more general as an open bounded set in $\mathbb{R}^N$, is that we would use the compact imbedding theorem for fractional sobolev space $W^{s,p}(\Omega)$, which to our best knowledge, is valid only for domain (see \cite{DCKP,CLM}).
\end{Remark}

\begin{Remark}\label{Rem2}
For recent application of $\Gamma$-convergence in other situations of the fractional case, one can refer to \cite{ADPM,BPS,Ponce}. For a general introduction of $\Gamma$-convergence, one can refer to \cite{Braides, DalMaso}.
\end{Remark}

\begin{Remark}\footnote{This notification is attributed to Professor Terasawa.}
One can also find a similar result in \cite{BPLS}, which established a approximation to H\"{o}lder infinity Laplacian equation by $\Gamma$-convergence by Orlicz fractional Laplacians (see Theorem 5.2 therein).
\end{Remark}

We may apply the $\Gamma$-limit of "free" energy results above to minimum of the form
\begin{equation}\label{MiniEn}
m_{\epsilon}=\inf
\left\{
\int_{\Omega}f_{\epsilon}(x,u,D^su)dx-\int_{\Omega}\langle g,u\rangle dx:u=\varphi\ on\ \partial\Omega
\right\},
\end{equation}
during which $\Omega$ stands for a bounded (smooth enough) domain of $\mathbb{R}^N$ and $s\in(0,1)$, and $D^s$ denotes a fractional differential operator. 

Applications of $\Gamma$-convergence to PDEs can be generally related to the behavior of the Euler-Lagrange equations. Notice that the possibility of defining a $\Gamma$-limit related to these problems will not be linked to the properties (or even the existence) of the solutions of the related Euler-Lagrange equations (\cite{Braides}). For example, for fractional Laplacian equation $(-\Delta)^su=0$ in $\Omega$, the Dirichlet boundary problem $u=g$ on ${\partial\Omega}$ is ill-posedness, and the case $u=g$ on $\mathbb{R}^N\setminus\Omega$ is well-posedness, which means the boundary value is not only determined only in the domain $\Omega$, but the whole space (see \cite{GM,Warma}), but we can establish the existence and uniqueness of the minimizer for the fractional $p$-functional (\ref{FracFunc}) under the first Dirichlet condition. See also section \ref{svaryingprocess}.

So the uniqueness of the minimizer of the limitation energy does not imply corresponding uniqueness of the solutions to the limitation Euler-Lagrange equations (Thm.\ref{ParUniqueness}). In particular in this paper, we can only state that the minimizer sequence would convergence to a minimizer of the limitation functional, but there are also many other extensions and characterizations of the minimizer as the weak solutions to different Euler-Lagrange forms (see, e.g., \cite{Barron, CEG, CLM}).

\medskip

We can see that in the functional (\ref{MiniEn}) there exist two other terms: the force term $g$ and the boundary $\varphi$.
Anyway even if we have established the $\Gamma$-convergence for the functional $E_{\alpha,p}(u)$, we can only get immediately the same convergence result for the minimizers of such functionals in the same space, but not for the minimum problems with non-homogeneous Dirichlet boundary conditions. So we have to verify the compatibility of the condition $u=\varphi$ on $\partial\Omega$, which is our next main result in this paper.

For the preparation of  the investigation of the compatibility of the Dirichlet boundary conditions, we give some definitions first. Let $\Omega$ be a bounded Lipschitz domain in $\mathbb{R}^N$, $0<\alpha<1$ and $p>\frac{2N}{\alpha}$. Now with $\varphi\in C^{0,\alpha}(\Omega)$ we define some admissible function sets
$$
X^{\varphi}_{\alpha,p}(\Omega):=\{u:\left(\int_{\Omega\times\Omega}\frac{|u(x)-u(y)|^p}{|x-y|^{\alpha p}}dxdy\right)^{\frac{1}{p}}<+\infty,\ u=\varphi\ on\ \partial\Omega\},
$$
and
$$
X^{\varphi}_{\alpha,\infty}(\overline\Omega):=
\{u:\sup\limits_{x,y\in\overline{\Omega},x\neq y}\frac{|u(x)-u(y)|}{|x-y|^{\alpha}}<+\infty,\ u=\varphi\ on\ \partial\Omega\}.
$$

The energy integrals are defined as follows:
\begin{equation}\label{BoundaryP}
E^{\varphi}_{\alpha,p}(u)=
\left\{
\begin{array}{lr}
\left(\int_{\Omega\times\Omega}\frac{|u(x)-u(y)|^p}{|x-y|^{\alpha p}}dxdy\right)^{\frac{1}{p}},& if\ u\in X^{\varphi}_{\alpha,p}(\Omega),\\
\infty& otherwise;
\end{array}
\right.
\end{equation}
and
\begin{equation}\label{BoundaryInfty}
\overline{E}^{\varphi}_{\alpha,\infty}(u)=
\left\{
\begin{array}{lr}
\sup\limits_{x,y\in\overline{\Omega},x\neq y}\frac{|u(x)-u(y)|}{|x-y|^{\alpha}},& if\ u\in X^{\varphi}_{\alpha,\infty}(\overline\Omega),\\
\infty& otherwise.
\end{array}
\right.
\end{equation}

Since when $p$ is large enough, we have $W^{s,p}(\Omega)$ imbedded in $C^{0,s-\frac{N}{p}}(\overline\Omega)$ compactly, so functions in $W^{s,p}(\Omega)$ become continuous automatically up to the boundary. Then on the existence and uniqueness of minimizers for functionals $E^{\varphi}_{\alpha,p}(u_p)$ ($p>\frac{2N}{\alpha}$), one can refer to (\cite{CLM}, Lemma 6.3). For the completeness, we state the lemma here without proof.
\begin{Lemma}[\cite{CLM} Lemma 6.3, existence and uniqueness of minimizer]\label{CLMlemma6.3}
let $\alpha\in(0,1]$ and assume that $\Omega$ is a bounded Lipschitz domain. Consider $\varphi\in C^{0,\alpha}(\partial\Omega)$ and define the set
$$
X_{\varphi}(\Omega):=\{u\in C(\overline{\Omega}),\ u=\varphi\ on\ \partial\Omega\}.
$$
Define the minimization problem
$$
I=\inf\limits_{u\in X_{\varphi}(\Omega)}E_p(u),
$$
where
$$
E_p(u)=\int_{\Omega\times\Omega}|\frac{u(x)-u(y)}{|x-y|^{\alpha}}|^pdxdy.
$$
Then for any $p>\frac{2N}{\alpha}$, problem $I$ has a unique minimizer $u_p$. Moreover, for any function
$\phi\in C^{\infty}_c(\Omega)$, we have
$$
\begin{array}{rcl}
\int_{\Omega\times\Omega}|\frac{u_p(x)-u_p(y)}{|x-y|^{\alpha}}|^{p-1}
\left\{
\frac{sgn(u_p(y)-u_p(x))}{|y-x|^{\alpha}}
\right\}
(\phi(y)-\phi(x))dxdy=0.
\end{array}
$$
\end{Lemma}

Now we give another main result in this section:
\begin{Theorem}[Compatibility of Dirichlet boundary]\label{CompDirichlet}
Let $\alpha\in(0,1)$ and $q>\frac{2N}{\alpha}$. Let $\Omega$ be a bounded Lipschitz domain in $\mathbb{R}^N$ and $\varphi\in C^{0,\alpha}(\partial\Omega)$. Then we have\\
\indent(i) $\Gamma(L^q(\Omega))$-$\lim\limits_{p\rightarrow+\infty}E^{\varphi}_{\alpha,p}=\overline{E}^{\varphi}_{\alpha,\infty}$;\\
\indent(ii)\footnote{Many thanks to Professor Terasawa for pointing out the already existing reference on this result. See \cite{JLJ} Corollary 6.1.1, which is more general.} Let $\{u_p\}_p$ is the minimizer sequence of the functional sequence $\{E^{\varphi}_{\alpha,p}\}_p$. If $u_p\rightarrow u$ in $L^q(\Omega)$ strongly and
$$
\Gamma(L^q(\Omega))-\lim\limits_{p\rightarrow+\infty}E^{\varphi}_{\alpha,p}=\overline{E}^{\varphi}_{\alpha,\infty},
$$
then $u$ is a minimizer of $\overline{E}^{\varphi}_{\alpha,\infty}$ in $X^{\varphi}_{\alpha,\infty}(\overline\Omega)$.
\end{Theorem}

\subsection{Proof of Theorem \ref{GammaConvinfty}}

\begin{Proposition}[$\Gamma-\limsup$ inequality]\label{limsup}
Let $\alpha\in(0,1)$ and $q>\frac{N}{\alpha}$, and let $u\in L^q(\Omega)$. Let $\{p\}$ be a sequence of strictly increasing positive numbers going to $+\infty$. Then there exists a sequence $\{u_p\}_p$ converging to $u$ in $L^q(\Omega)$ such that
$$
\limsup_{p\rightarrow+\infty}E_{\alpha,p}(u_p)\leq E_{\alpha,\infty}(u)\leq \overline{E}_{\alpha,\infty}(u).
$$
\end{Proposition}

\begin{proof}
If $\overline{E}_{\alpha,\infty}(u)=+\infty$, the inequality is satisfied automatically, so there is noting to prove. Thus let us take $\overline{E}_{\alpha,\infty}(u)<+\infty$.

Now we will find a "recovery sequence" to verify the condition (ii) of the $\Gamma$-convergence equality.
Let us consider the sequence $\{u_p\}_p\subset L^q(\Omega)$, where $u_p:=u$ for all $p\geq 1$. Then we have
$$
\begin{array}{lcl}
\limsup\limits_{p\rightarrow+\infty}(\int_{\Omega\times\Omega}\frac{|u(x)-u(y)|^p}{|x-y|^{\alpha p}}dxdy)^{\frac{1}{p}}\\
\leq\lim\limits_{p\rightarrow+\infty}
\left(\int_{\Omega\times\Omega}(\sup\limits_{x\neq y,x,y\in\Omega}\frac{|u(x)-u(y)|}{|x-y|^{\alpha}})^pdxdy\right)^{\frac{1}{p}},
\end{array}
$$
and then
$$
\begin{array}{rcl}
\limsup\limits_{p\rightarrow+\infty}(\int_{\Omega\times\Omega}\frac{|u(x)-u(y)|^p}{|x-y|^{\alpha p}}dxdy)^{\frac{1}{p}}
&\leq&\sup\limits_{x\neq y,x,y\in\Omega}\frac{|u(x)-u(y)|}{|x-y|^{\alpha}}=E_{\alpha,\infty}(u)\\
&\leq&\sup\limits_{x\neq y,x,y\in\overline{\Omega}}|\frac{u(x)-u(y)}{|x-y|^{\alpha}}|=\overline{E}_{\alpha,\infty}(u),
\end{array}
$$
which concludes the desired result.
\end{proof}

\medskip

Now we attempt to verify the condition (i) in the Definition (\ref{GammaConvDef}).

In this paper, we use $\mathcal{L}^N(U)$ to denote the $N$-dimensional Lebesgue measure of the measurable set $U\subset\mathbb{R}^N$.

\begin{Proposition}[$\Gamma-\liminf$ inequality]\label{liminf}
Let $\alpha\in(0,1)$, $q>\frac{N}{\alpha}$, $u\in L^q(\Omega)$, and let $\{p\}$ be a sequence of strictly increasing positive numbers going to $+\infty$. Consider any sequence $\{u_p\}_p\subset L^q(\Omega)$ converging to $u$ in $L^q(\Omega)$, then we have
$$
E_{\alpha,\infty}(u)\leq\liminf_{p\rightarrow+\infty}E_{\alpha,p}(u_p).
$$
And if $q>\frac{2N}{\alpha}$, we have
$$
\overline{E}_{\alpha,\infty}(u)\leq\liminf_{p\rightarrow+\infty}E_{\alpha,p}(u_p).
$$
\end{Proposition}

\begin{proof}
{\bf Step 1.}
If $\liminf\limits_{p\rightarrow+\infty}E_{\alpha,p}(u_p)=\infty$, the inequality is satisfied automatically, so there is nothing to prove.

Then let us suppose that $\liminf\limits_{p\rightarrow+\infty}E_{\alpha,p}(u_p)<\infty$. Then we can infer that there exists $L>0$ such that $\liminf\limits_{p\rightarrow+\infty}E_{\alpha,p}(u_p)$ is uniformly bounded, i.e.
$$
\liminf_{p\rightarrow +\infty}E_{\alpha,p}(u_p)\leq L,
$$
and based on a subsequence $\{u_{p_n}\}$ of $\{u_p\}$ we have
$$
\lim_{n\rightarrow +\infty}E_{\alpha,p_n}(u_{p_n})=\liminf_{p\rightarrow +\infty}E_{\alpha,p}(u_p)\leq L,
$$

For convenience we still denote the sequence $\{p_n\}$ by $\{p\}$.

Since $p\rightarrow+\infty$, then for $p$ large enough there holds $p>q$. Then by H\"{o}lder inequality, we have that
$$
\begin{array}{ccc}
\left(\int_{\Omega\times\Omega}\frac{|u_p(x)-u_p(y)|^q}{|x-y|^{\alpha q}}dxdy\right)^{1/q}\leq
C(\Omega,N)\left(\int_{\Omega\times\Omega}\frac{|u_p(x)-u_p(y)|^p}{|x-y|^{\alpha p}}dxdy\right)^{1/p}\\
\leq C(\Omega,N) L
\end{array}
$$

As $u_{p}\rightarrow u$ strongly in $L^q(\Omega)$ for every $p$, in view of Poincar\'{e}-Wirtinger inequality, we have that $u_{p}$ is uniformly bounded in $W^{\alpha-\frac{N}{q},q}(\Omega)$. Then we can extract a subsequence $u_{p}$ (not relabelled) such that
$$
u_{p}\rightharpoonup u
$$
weakly in $W^{\alpha-\frac{N}{q},q}(\Omega)$. Then by the lower semi-continuity of $W^{\alpha-\frac{N}{q},q}(\Omega)$ and H\"{o}lder inequality we have
$$
\begin{array}{rcl}
\int_{\Omega\times\Omega}|\frac{u(x)-u(y)}{|x-y|^{\alpha}}|^qdxdy&\leq&
\lim\limits_{n\rightarrow+\infty}\int_{\Omega\times\Omega}|\frac{u_{p}(x)-u_{p}(y)}{|x-y|^{\alpha}}|^qdxdy\\
&\leq&\lim\limits_{n\rightarrow+\infty}\mathcal{L}^N(\Omega)^{2}
\left(\int_{\Omega\times\Omega}|\frac{u_{p}(x)-u_{p}(y)}{|x-y|^{\alpha}}|^{p}dxdy\right)^{\frac{q}{p}}.
\end{array}
$$
This means
$$
\left(\int_{\Omega\times\Omega}|\frac{u(x)-u(y)}{|x-y|^{\alpha}}|^qdxdy\right)^{\frac{1}{q}}
\leq\mathcal{L}^N(\Omega)^{2/q}\lim\limits_{n\rightarrow+\infty}
\left(\int_{\Omega\times\Omega}|\frac{u(x)-u(y)}{|x-y|^{\alpha}}|^pdxdy\right)^{\frac{1}{p}}.
$$
Therefore letting $q\rightarrow+\infty$ we obtain
$$
E_{\alpha,\infty}(u)\leq\lim\limits_{p\rightarrow+\infty}E_{\alpha,p}(u_p),
$$
which implies for the original sequence $\{p\}$
$$
E_{\alpha,\infty}(u)\leq\liminf\limits_{p\rightarrow+\infty}E_{\alpha,p}(u_p).
$$

Then if $q>2N/\alpha$, then by Sobolev compact imbedding theorem, we have
$$
\|u\|_{C^{0,\gamma}(\overline{\Omega})}\leq C\|u\|_{W^{\alpha-\frac{N}{q},q}(\Omega)},
$$
where $\gamma=\alpha-\frac{2N}{q}$, and
$$
C^{0,\gamma}(\overline{\Omega}):=\{f\in C(\overline{\Omega}),\ \|f\|_{L^{\infty}(\overline{\Omega})}+\sup\limits_{x\neq y,x,y\in\overline\Omega}\frac{|f(x)-f(y)|}{|x-y|^{\gamma}}<+\infty\}.
$$
So $u$ is continuous up to the boundary, then for every boundary point $x_0$ on $\partial\Omega$, we can find a sequence $\{x_m\}\subset\Omega$ ($m\in\mathbb{N}$) such that $\lim\limits_{m\rightarrow+\infty}|x_m-x_0|=0$ and $|u(x_m)-u(x_0)|<\epsilon$ for $\forall$ $\epsilon>0$ when $m$ is large enough. Then for any $y\in\Omega$
$$
\begin{array}{llll}
\frac{|u(x_0)-u(y)|}{|x_0-y|^{\alpha}}&=&\frac{|u(x_0)-u(x_m)+u(x_m)-u(y)|}{|x_0-y|^{\alpha}}\\
&\leq&\frac{|u(x_0)-u(x_m)|}{|x_0-y|^{\alpha}}+\frac{|u(x_m)-u(y)|}{|x_0-y|^{\alpha}}\\
&\leq&\epsilon+\frac{|u(x_m)-u(y)|}{|x_0-y|^{\alpha}}\\
&\leq&\sup\limits_{x\neq y,x,y\in\Omega}|\frac{u(x)-u(y)}{|x-y|^{\alpha}}|+\epsilon.
\end{array}
$$
Since $\epsilon$ is arbitrary, we conclude that
$$
\sup\limits_{x\neq y,x,y\in\overline{\Omega}}|\frac{u(x)-u(y)}{|x-y|^{\alpha}}|\leq
\sup\limits_{x\neq y,x,y\in\Omega}|\frac{u(x)-u(y)}{|x-y|^{\alpha}}|,
$$
and obviously
$$
\sup\limits_{x\neq y,x,y\in\Omega}|\frac{u(x)-u(y)}{|x-y|^{\alpha}}|\leq
\sup\limits_{x\neq y,x,y\in\overline{\Omega}}|\frac{u(x)-u(y)}{|x-y|^{\alpha}}|,
$$
which concludes the desired result
$$
\overline{E}_{\alpha,\infty}(u)\leq\liminf_{p\rightarrow+\infty}E_{\alpha,p}(u_p).
$$

\end{proof}

\subsection{Proof of Theorem \ref{CompDirichlet}}
In this section, based on the $\Gamma$-convergence results established in Theorem \ref{GammaConvinfty}, we verify the compatibility of non-homogeneous of Dirichlet condition for the fractional Laplacian functional (\ref{FracFunc}) as $p\rightarrow+\infty$.

\begin{proof}
{\bf Step 1.} Now we firstly verify the $\liminf$ inequality in the definition of $\Gamma$-convergence for
$$
\Gamma(L^q(\Omega))-\lim\limits_{p\rightarrow+\infty}E^{\varphi}_{\alpha,p}=E^{\varphi}_{\alpha,\infty}.
$$
So let $u_p\rightarrow u$ in $L^q(\Omega)$. Then if
$$
\liminf\limits_{p\rightarrow+\infty}E^{\varphi}_{\alpha,p}(u_p)=+\infty,
$$
there is nothing to prove. So we may directly assume that for a sequence $\{u_p\}\subset L^q(\Omega)$ such that
$$
\liminf\limits_{p\rightarrow+\infty}E^{\varphi}_{\alpha,p}(u_p)<+\infty.
$$
Then we can extract a subsequence (not relabelled) $\{u_p\}$ and there exists some $L>0$ such that
$$
\lim\limits_{p\rightarrow+\infty}E^{\varphi}_{\alpha,p}(u_p)<L,
$$
which implies that sequence $\{u_p\}$ is uniformly bounded in $W^{\alpha-\frac{N}{q},q}(\Omega)$ by Sobolev imbedding theorem.

Then as in the step 1 of proof to Proposition \ref{liminf}, since $u_p\rightarrow u$ in $L^q(\Omega)$ strongly, by Poincar\'{e}-Wirtinger inequality, we have $u_p\rightharpoonup u$ weakly in $W^{\alpha-\frac{N}{q},q}(\Omega)$. Since for any $p\in(q,+\infty)$, $u_p\in X^{\varphi}_{\alpha,p}(\Omega)$ and by Sobolev imbedding theorem $X^{\varphi}_{\alpha,p}(\Omega)\hookrightarrow X^{\varphi}_{\alpha,q}$, we infer that $\{u_p\}\subset X^{\varphi}_{\alpha,q}$ for $p>q$. Since $X^{\varphi}_{\alpha,q}$ is closed in $W^{\alpha-\frac{N}{q},q}(\Omega)$, then by the reflexivity of $W^{\alpha-\frac{N}{q},q}(\Omega)$, we get that $u\in X^{\varphi}_{\alpha,q}(\Omega)$. Then following the same process as step 2 in the proof of Proposition \ref{liminf}, we have that for any sequence $u_p\rightarrow u$ in $L^q(\Omega)$
$$
\overline{E}^{\varphi}_{\alpha,\infty}(u)\leq\liminf\limits_{p\rightarrow+\infty}E^{\varphi}_{\alpha,p}(u_p).
$$
The only difference from Proposition \ref{liminf} is that the function space $X^{\varphi}_{\alpha,q}(\Omega)$: since $q>\frac{2N}{\alpha}$, by Sobolev imbedding theorem $X^{\varphi}_{\alpha,q}(\Omega)\hookrightarrow C(\overline\Omega)$, so we can directly get the estimates for $\overline{E}^{\varphi}_{\alpha,\infty}(u)$.

{\bf Step 2.} Now we are in the position to verify the recovery sequence condition in the $\Gamma$-convergence definition for
$$
\Gamma(L^q(\Omega))-\lim\limits_{p\rightarrow+\infty}E^{\varphi}_{\alpha,p}=\overline{E}^{\varphi}_{\alpha,\infty}.
$$
that is, to find a sequence $\{u_p\}\subset X^{\varphi}_{\alpha,p}(\Omega)$ such that for any $u\in L^q(\Omega)$
\begin{align}\label{supineq}
\begin{array}{rcl}
\overline{E}^{\varphi}_{\alpha,\infty}(u)\geq\limsup\limits_{p\rightarrow+\infty}E^{\varphi}_{\alpha,p}(u_p),
\end{array}
\end{align}
and
$$
u_p\rightarrow u\ in\ L^q(\Omega).
$$

In fact, as in Proposition \ref{limsup}, we can directly let $u_p:=u$ in $X^{\varphi}_{\alpha,p}(\Omega)$. Indeed, as $u\in X^{\varphi}_{\alpha,\infty}(\overline\Omega)$, then $u\in X_{\varphi}(\Omega)$ (defined in Lemma \ref{CLMlemma6.3}), and by Sobolev imbedding theorem we infer that $E_{\alpha,p}(u)$ is bounded. Because $u\in X^{\varphi}_{\alpha,\infty}(\overline\Omega)$, which is up to the boundary, and when $p>\frac{2N}{\alpha}$, $u\in C^{0,\alpha-\frac{2N}{p}}(\overline\Omega)$, we can infer that $u\in X^{\varphi}_{\alpha,p}(\Omega)$. Then
$$
\begin{array}{llll}
\left\{
\limsup\limits_{p\rightarrow+\infty}\int_{\Omega\time\Omega}\frac{|u(x)-u(y)|^p}{|x-y|^{\alpha p}}dxdy
\right\}^{1/p}&\leq&\limsup\limits_{p\rightarrow+\infty}
(\mathcal{L}^N(\Omega))^{\frac{2}{p}}\left\{[u]^p_{W^{\alpha,\infty}(\overline\Omega)}\right\}^{1/p}\\
&=&[u]_{W^{\alpha,\infty}(\overline\Omega)},
\end{array}
$$
which concludes the results together with step 1.

{\bf Step 3.} Next we prove (ii) of Theorem \ref{CompDirichlet}. We claim that $u\in X^{\varphi}_{\alpha,\infty}(\overline\Omega)$ is a minimizer of functional $\overline{E}^{\varphi}_{\alpha,\infty}$ given any $v\in X^{\varphi}_{\alpha,\infty}(\overline\Omega)$.

Supposing that the sequence $\{E^{\varphi}_{\alpha,p}\}_p$ $\Gamma(L^q(\Omega))$-converges to $\overline{E}^{\varphi}_{\alpha,\infty}$ at $v$, then by the definition of $\Gamma$-convergence, there exists a sequence $\{\omega_p\}_p$ such that
$$
\omega_p\rightarrow v\ in\ L^q(\Omega),\ as\ p\rightarrow+\infty,
$$
and
\begin{align}\label{sequencev}
\limsup\limits_{p\rightarrow+\infty}E^{\varphi}_{\alpha,p}(\omega_p)\leq E^{\varphi}_{\alpha,\infty}(v).
\end{align}
Since by assumption the sequence $\{u_p\}_p$ are the minimizers of $E^{\varphi}_{\alpha,p}$ in $X_{\varphi}(\Omega)$ for corresponding $p$, we infer that
$$
E^{\varphi}_{\alpha,p}(u_p)\leq E^{\varphi}_{\alpha,p}(\omega_p).
$$
Thus we have
\begin{align}\label{upvsomegap}
\liminf\limits_{p\rightarrow+\infty}E^{\varphi}_{\alpha,p}(u_p)\leq
\limsup\limits_{p\rightarrow+\infty}E^{\varphi}_{\alpha,p}(u_p)\leq
\limsup\limits_{p\rightarrow+\infty}\overline{E}^{\varphi}_{\alpha,\infty}(\omega_p).
\end{align}
So combining (\ref{sequencev}) and (\ref{upvsomegap}) yields that
$$
\overline{E}^{\varphi}_{\alpha,\infty}(u)\leq\overline{E}^{\varphi}_{\alpha,\infty}(v),
$$
which concludes the proof.
\end{proof}

\section{Asymptotic Behaviour on Varying $s$}\label{svaryingprocess}

Let $0<s<1$, $p\in(1,+\infty)$, and $\Omega\subset\mathbb{R}^N$ be an open bounded set. We consider the nonlocal nonlinear operator $(-\Delta_p)^su$ interpreted as
$$
(-\Delta_p)^su(x):=2\lim_{\epsilon\searrow0}
\int_{\mathbb{R}^N\setminus B_{\epsilon}(x)}\frac{|u(x)-u(y)|^{p-2}(u(x)-u(y))}{|x-y|^{N+sp}}dy,\ x\in\mathbb{R}^N.
$$
For more information on this operator we refer the reader to \cite{ADPM,BLP,BPS,DCKP,FP,LL,Palatucci}.

Firstly, we give a glimpse of the operator $(-\Delta_p)^s$ acting on the space $W^{s,p}_0(\Omega)$, which is defined as the closure of $C^{\infty}_0(\Omega)$ under the semi-norm
$$
[u]_{W^{s,p}(\Omega)}=\int_{\Omega\times\Omega}\frac{|u(x)-u(y)|^p}{|x-y|^{N+sp}}dxdy,
$$
which is in fact also a norm in $W^{s,p}_0(\Omega)$. If $p=2$ and $f=0$, the operator becomes liner case $(-\Delta)^s$ and the corresponding equations is fractional Laplacian, denoted as
\begin{equation}\label{fraclap}
\left\{
\begin{array}{lr}
(-\Delta)^su=0,&in\ \Omega,\\
u=g,&on\ \partial\Omega.
\end{array}
\right.
\end{equation}
However, by \cite{GM} the equation (\ref{fraclap}) is not well-posedness corresponding to the well-posedness non-homogeneous Dirichlet boundary condition given by
\begin{equation}\label{fracplapwider}
\left\{
\begin{array}{lr}
(-\Delta)^su=0,&in\ \Omega,\\
u=g,&on\ \mathbb{R}^N\setminus\Omega.
\end{array}
\right.
\end{equation}
In other words, the ill-posedness of (\ref{fraclap}) and the well-posedness of (\ref{fracplapwider}) show that an $(-\Delta)^s$ function in a domain $\Omega$ cannot be determined only by its value on the boundary $\partial\Omega$, but depends on its value on the whole area $\mathbb{R}^N\setminus\Omega$. For more details in this direction, one can see such as \cite{BB,BH,ROS1,ROS2,ROS3,ROS4} etc. and references therein.

Then based on the information above, we utilize the admissible space for the operator $(-\Delta_p)^s$, the nonlocal Sobolev space $\bf\widetilde{W}^{s,p}_0(\Omega)$ defined as the closure of $C^{\infty}_0(\Omega)$ under the semi-norm
$$
[u]_{W^{s,p}(\mathbb{R}^N)}=\int_{\mathbb{R}^N\times\mathbb{R}^N}\frac{|u(x)-u(y)|^p}{|x-y|^{N+sp}}dxdy,
$$
which is in fact also a norm in $\widetilde{W}^{s,p}_0(\Omega)$. And if the boundary $\partial\Omega$ regular enough, such as Lipschitz, the space $\widetilde{W}^{s,p}_0(\Omega)$ is in coincidence with $W^{s,p}_0(\Omega)$, i.e., $\Omega$ can be extensible. For more information on this topic, one can refer to \cite{Zhou}. And in this section, we do not assume any regularity on $\partial\Omega$.

However, if we work in the space $\widetilde{W}^{s,p}_0(\Omega)$, we would find that it seems a little difficult to get uniform comparison estimations for a pair of $s$ and $s'$, not to mention a sequence of $s_j$.

Then for our special problem setting here, we utilize a $relative$-$nonlocal$ Sobolev space denoted as $\bf\widetilde{W}^{s,p}_{0,tR}(\Omega)$, in which, $t$ is large than 1, and $R$ denotes the diameter of $\Omega$, defined by
$$
R:=\sup_{x,y\in\Omega}\{|x-y|:\forall\ x,y\in\Omega\}.
$$

First we define the semi-norm
\begin{align}\label{nonhomseminorm}
[u]_{W^{s,p}_{tR}(\Omega)}:=
\left\{\int_{B_{tR}(\Omega)\times B_{tR}(\Omega)}\frac{|u(x)-u(y)|^p}{|x-y|^{N+sp}}dxdy\right\}^{\frac{1}{p}}
\end{align}
for any measurable function $u$ in $L^p(\Omega)$, in which, $B_{tR}(\Omega)$ is define as the $N$-dimensional ball with diameter $tR$ located at the same center as the smallest ball containing $\Omega$.

Now we define the $\bf relative$-$\bf nonlocal$ Sobolev space $\bf\widetilde{W}^{s,p}_{0,tR}(\Omega)$ as the completion of $C^{\infty}_0(\Omega)$ with respect to the semi-norm (\ref{nonhomseminorm})
\begin{align}\label{workspace}
\|u\|_{\widetilde{W}^{s,p}_{0,tR}(\Omega)}:=
\left\{\int_{B_{tR}(\Omega)\times B_{tR}(\Omega)}\frac{|u(x)-u(y)|^p}{|x-y|^{N+sp}}dxdy
\right\}^{\frac{1}{p}},\ \forall\ u\in C^{\infty}_0(\Omega).
\end{align}
This is a reflexive Banach space for $1<p<+\infty$ and $0<s<1$. $t$ is independent of $\Omega$.

In fact, this space is equivalent to $\widetilde{W}^{s,p}_{0}(\Omega)$ (see Appendix in \cite{Li}). Also this is a direct result by the sufficient and necessary condition for extensible domain (see \cite{Zhou}). Since the ball $B_{tR}(\Omega)$ clearly fits for the condition in \cite{Zhou}, then we can directly conclude that $\|u\|_{\widetilde{W}^{s,p}_{0}(\Omega)}\leq C\|u\|_{\widetilde{W}^{s,p}_{0,tR}(\Omega)}$. For more information on the space $\widetilde{W}^{s,p}_{0,tR}(\Omega)$, one can see our another paper \cite{Li}.

\subsection{$\Gamma$-convergence as $s_j\rightarrow s$ from Above}

In this subsection, we use $\Gamma$-convergence to investigate the asymptotic behaviour of the following equations with varying $s$,
\begin{equation}\label{fracplap}
\left\{
\begin{array}{lr}
(-\Delta_p)^su=f,&in\ \Omega,\\
u=0,&on\ B_{tR}(\Omega)\setminus\Omega,
\end{array}
\right.
\end{equation}
in the weak sense as
$$
\left\{
\begin{array}{lll}
u\in\widetilde{W}^{s,p}_{0,tR}(\Omega),\\
\int_{B_{tR}(\Omega)\times B_{tR}(\Omega)}\frac{|u(x)-u(y)|^{p-s}(u(x)-u(y))(v(x)-v(y))}{|x-y|^{N+sp}}dxdy\\
=\int_{\Omega}fvdx,\ for\ every\ v\in\widetilde{W}^{s,p}_{0,tR}(\Omega),
\end{array}
\right.
$$
for which, the variational form is
$$
\min_{u\in \widetilde{W}^{s,p}_{0,tR}(\Omega)}\left(\frac{1}{p}
\int_{B_{tR}(\Omega)\times B_{tR}(\Omega)}\frac{|u(x)-u(y)|^p}{|x-y|^{N+sp}}dxdy+\int_{\Omega}fudx\right).
$$
For the existence and uniqueness of solutions to this equation, one can refer to \cite{DCKP,FP,LL}. In fact, it is a very standard
approach based on the direct method and strict convexity of the semi-norm $W^{s,p}_{tR}(\Omega)$.

For every $0<s<1$ and $p\in(1,+\infty)$, let us define the functional $F_{s}(u)$ as
$$
F_s(u)=\frac{1}{p}\int_{B_{tR}(\Omega)\times B_{tR}(\Omega)}\frac{|u(x)-u(y)|^p}{|x-y|^{N+sp}}dxdy+\int_{\Omega}fudx.
$$

Let $\forall0<\epsilon\ll1$, $0<s<1$ and $p\in(1,+\infty)$, and let $(\widetilde{W}^{s+\epsilon,p}_{0,tR}(\Omega))^*$ denote the usual dual space of the space $\widetilde{W}^{s+\epsilon,p}_{0,tR}(\Omega)$. Since $\widetilde{W}^{s+\epsilon,p}_{0,tR}(\Omega)\hookrightarrow \widetilde{W}^{s,p}_{0,tR}(\Omega)$, we have
$$
(\widetilde{W}^{s,p}_{0,tR}(\Omega))^*\hookrightarrow(\widetilde{W}^{s+\epsilon,p}_{0,tR}(\Omega))^*.
$$

A sequence $\{F_k\}_k$ is said to be $equi-coercive$ if there exist a compact set $K\subset X$ such that
$$
\inf_XF_k=\inf_KF_k
$$
for each $k\in\mathbb{N}$ (see \cite{DalMaso,Braides}).

In the following theorem, we give the $\Gamma$-convergence on functionals $F_{s}(u)$.

\begin{Theorem}\label{sconvaboveThm}
Let $\Omega\subset\mathbb{R}^N$ be a bounded open set, $0<s<1$, $p\in(1,+\infty)$, and for $f\in(\widetilde{W}^{s,p}_{0,tR}(\Omega))^*$. If $\{s_j\}_j\subset(0,1)$ be non-increasing sequence converging to $s$, then the sequence $\{F_{s_j}\}_j$ defined on $L^p(\Omega)$ is equi-coercive in $L^p(\Omega)$, and $F_{s_j}(u)$ $\Gamma$-converges to $F_s(u)$ in $L^p(\Omega)$ at every $u\in L^p(\Omega)$ which satisfies
$$
\int_{B_{tR}(\Omega)\times B_{tR}(\Omega)}\frac{|u(x)-u(y)|^p}{|x-y|^{N+(s+\epsilon)p}}dxdy<+\infty.
$$
\end{Theorem}

\begin{proof}
We observe that it is obviously that the infimum of each $F_{s_j}$ is attained in $\widetilde{W}^{s_j,p}_{0,tR}(\Omega)$. It is well known that the Fr\'{e}chet derivative of $F_{s_j}$ (i.e. Euler-Lagrange forms) is the functional on $\widetilde{W}^{s_j,p}_{0,tR}(\Omega)$ given by
$$
v\rightarrow\int_{B_{tR}(\Omega)\times B_{tR}(\Omega)}\frac{|u(x)-u(y)|^{p-2}(u(x)-u(y))(v(x)-v(y))}{|x-y|^{N+s_jp}}dxdy+\langle f,v\rangle,
$$
in which $\langle f,v\rangle$ denotes the usual dual product. The unique minimizer of $F_{s_j}$ in $\widetilde{W}^{s_j,p}_{0,tR}(\Omega)$ is just the solution $u_{s_j}$ to (\ref{FracLapEqu}) (see e.g. \cite{FP,LL,Palatucci}). Due to the Rellich-type embedding theorems, the closure $\overline{K}$ in $L^p(\Omega)$ of the set $K:=\{u_{s_j},j\in\mathbb{N}\}$ is compact. Again from the discussion above we infer that
$$
\inf_{L^p(\Omega)}F_{s_j}=F(u_{s_j})=\inf_{\overline{K}}F_{s_j}
$$
for each $j\in\mathbb{N}$. Then the sequence $\{F_{s_j}\}_j$ is equi-coercive. For more equivalent conditions on equi-coercive, one can refer to Chapter 2 and 7 in \cite{DalMaso}.

Now consider a sequence $\{w_{s_j}\}_j$ in $L^p(\Omega)$ that converges to $w$ in $L^p(\Omega)$. If
$$
\liminf_{j\rightarrow+\infty}\left(\frac{1}{p}
\int_{B_{tR}(\Omega)\times B_{tR}(\Omega)}\frac{|w_{s_j}(x)-w_{s_j}(y)|^p}{|x-y|^{N+s_jp}}dxdy+\langle f,w_{s_j}\rangle
\right)<+\infty,
$$
one can extract a subsequence (not relabelled) $\{w_{s_j}\}_j$ for which
$$
\begin{array}{lllll}
\lim\limits_{j\rightarrow+\infty}\frac{1}{p}
\int_{B_{tR}(\Omega)\times B_{tR}(\Omega)}\frac{|w_{s_j}(x)-w_{s_j}(y)|^p}{|x-y|^{N+s_jp}}dxdy+\langle f,w_{s_j}\rangle\\
=\liminf\limits_{j\rightarrow+\infty}\left(\frac{1}{p}
\int_{B_{tR}(\Omega)\times B_{tR}(\Omega)}\frac{|w_{s_j}(x)-w_{s_j}(y)|^p}{|x-y|^{N+s_jp}}dxdy+\langle f,w_{s_j}\rangle
\right)\\
=L<+\infty.
\end{array}
$$
Since $s\leq s_j$ and $f\in(\widetilde{W}^{s,p}_{0,tR}(\Omega))^*$, one can easily have that
$$
\frac{1}{p}\int_{B_{tR}(\Omega)\times B_{tR}(\Omega)}\frac{|w_{s_j}(x)-w_{s_j}(y)|^p}{|x-y|^{N+s_jp}}dxdy\leq
C\left(1+\|f\|_{(\widetilde{W}^{s_j,p}_{0,tR}(\Omega))^*}\|w_{s_j}\|_{\widetilde{W}^{s_j,p}_{0,tR}(\Omega)}\right)
$$
for some positive constant $C$ and each $j\in\mathbb{N}$. Then by Young's inequality we have that the sequence $\{w_{s_j}\}$ is uniformly bounded in $\widetilde{W}^{s,p}_{0,tR}(\Omega)$ by Sobolev-type embedding (see \cite{DNPV,Li}). Then thanks to the reflexivity of the space $\widetilde{W}^{s,p}_{0,tR}(\Omega)$ together with again the Sobolev-type embedding we have that $w\in \widetilde{W}^{s,p}_{0,tR}(\Omega)$. Then without loss of generality one can consider the weak convergence in $\widetilde{W}^{s,p}_{0,tR}(\Omega)$
$$
w_{s_j}\rightharpoonup w.
$$

For simplicity, we denote the diameter of $\Omega$ by $R$. Then by the weak lower semi-continuity one has
$$
\begin{array}{lrrrr}
\frac{1}{p}\int_{B_{tR}(\Omega)\times B_{tR}(\Omega)}\frac{|w(x)-w(y)|^p}{|x-y|^{N+sp}}dxdy\\
\leq\liminf\limits_{j\rightarrow+\infty}
\frac{1}{p}\int_{B_{tR}(\Omega)\times B_{tR}(\Omega)}\frac{|w_{s_j}(x)-w_{s_j}(y)|^p}{|x-y|^{N+sp}}dxdy\\
\leq\liminf\limits_{j\rightarrow+\infty}\frac{1}{p}{tR}^{(s_j-s)p}
\int_{B_{tR}(\Omega)\times B_{tR}(\Omega)}\frac{|w_{s_j}(x)-w_{s_j}(y)|^p}{|x-y|^{N+s_jp}}dxdy\\
=\liminf\limits_{j\rightarrow+\infty}(\frac{1}{p}{tR}^{(s_j-s)p}
\int_{B_{tR}(\Omega)\times B_{tR}(\Omega)}\frac{|w_{s_j}(x)-w_{s_j}(y)|^p}{|x-y|^{N+s_jp}}dxdy\\
+\langle f,w_{s_j}\rangle-\langle f,w_{s_j}\rangle)=L-\langle f,w\rangle.
\end{array}
$$

In fact, if we check the process above carefully, let $F_s(w)=+\infty$, then
$$
\liminf\limits_{j\rightarrow+\infty}\frac{1}{p}
\int_{B_{tR}(\Omega)\times B_{tR}(\Omega)}\frac{|w_{s_j}(x)-w_{s_j}(y)|^p}{|x-y|^{N+s_jp}}dxdy+\langle f,w_{s_j}\rangle=+\infty;
$$
if it is this case, then obviously
$$
F_s(w)=+\infty=
\liminf_{j\rightarrow+\infty}\frac{1}{p}
\int_{B_{tR}(\Omega)\times B_{tR}(\Omega)}\frac{|w_{s_j}(x)-w_{s_j}(y)|^p}{|x-y|^{N+s_jp}}dxdy+\langle f,w_{s_j}\rangle.
$$
So it follows from above arguments that if $w_{s_j}\rightarrow w$ in $L^p(\Omega)$, then we have
$$
F_{s_j}(w_{s_j})\rightarrow F_s(w)\ in\ [0,+\infty].
$$

We can complete the proof of the $\Gamma$-convergence by observing that for each $u\in L^p(\Omega)$,
$$
\begin{array}{l}
\lim\limits_{j\rightarrow+\infty}F_{s_j}(u)=\lim\limits_{j\rightarrow+\infty}\frac{1}{p}
\int_{B_{tR}(\Omega)\times B_{tR}(\Omega)}\frac{|u(x)-u(y)|^p}{|x-y|^{N+s_jp}}dxdy+\langle f,u\rangle\\
=\frac{1}{p}\int_{B_{tR}(\Omega)\times B_{tR}(\Omega)}\frac{|u(x)-u(y)|^p}{|x-y|^{N+sp}}dxdy+\langle f,u\rangle=F_s(u).
\end{array}
$$
\end{proof}

\subsection{$\Gamma$-convergence as $s_j\rightarrow s$ from Below}

In this subsection, we just give the $\Gamma$-convergence result of some free functionals to express some special characters of the asymptotic behaviours from below. We can see there is something different from the case converging from above. And in order to make the difference clear, we do not use the $relative$-$nonlocal$ setting. On the other hand, we would not use the Rellich-type compact embedding property, which needs the extension assumption of $\partial\Omega$ (see \cite{Zhou}), so we can be away from the relative-nonlocal setting for a while. We just work on the usual Sobolev space $W^{s,p}(\Omega)$, and this does not change the intrinsic quality.

Now we should make some modifications on the space we work on. Let $1<p<+\infty$, $0<s<1$ and $\Omega\subset\mathbb{R}^N$ be an open bounded set. We set
$$
W^{s^-,p}_0(\Omega)=W^{s,p}(\Omega)\cap\left(\bigcap_{0<k<s}W^{k,p}_0(\Omega)\right)
=\bigcap_{0<k<s}\left(W^{s,p}(\Omega)\cap W^{k,p}_0(\Omega)\right),
$$
where $W^{s,p}_0(\Omega)$ is the complete closure of $C^{\infty}_0(\Omega)$ under the semi-norm $W^{s,p}(\Omega)$ defined by
$$
[u]_{W^{s,p}(\Omega)}:=\int_{\Omega\times\Omega}\frac{|u(x)-u(y)|^p}{|x-y|^{N+sp}}dxdy.
$$
We can clearly see that $W^{s^-,p}_0(\Omega)$ is a closed vector space of $W^{s,p}(\Omega)$ satisfying
$$
W^{s,p}_0(\Omega)\subset W^{s^-,p}_0(\Omega).
$$

We define two functionals $\mathcal{E}^s$ and $\underline{\mathcal{E}}^s$, mapping $L^p(\Omega)$ to $[0,+\infty]$ as
$$
\begin{array}{lll}
\mathcal{E}^s=\left\{
\begin{array}{lr}
\int_{\Omega\times\Omega}\frac{|u(x)-u(y)|^p}{|x-y|^{N+sp}}dxdy, & if\ u\in W^{s,p}_0(\Omega),\\
+\infty & otherwise.
\end{array}
\right.\\

\underline{\mathcal{E}}^s=\left\{
\begin{array}{lr}
\int_{\Omega\times\Omega}\frac{|u(x)-u(y)|^p}{|x-y|^{N+sp}}dxdy, & if\ u\in W^{s^-,p}_0(\Omega),\\
+\infty & otherwise.
\end{array}
\right.
\end{array}
$$

For preparation we need the following definition.
\begin{Def}\label{relaxfun}
For every function $F:X\rightarrow\overline{\mathbb{R}}$ the $lower$ $semi$-$continuous$ $envelop$ (or $relaxed$ $function$) $sc^-F$ of $F$ is defined for every $x\in X$ by
$$
(sc^-F)(x)=\sup\limits_{G\in\mathcal{G}(F)}G(x),
$$
where $\mathcal{G}(F)$ is the set of all lower semi-continuous functions $G$ on $X$ such that $G(y)\leq F(y)$ for every $y\in X$.
\end{Def}
We can see that in fact $sc^-F$ is the greatest lower semi-continuous function majorized by $F$. For more information on the relax function and the relations with $\Gamma$-convergence function one can see Chapter $3-5$ in \cite{DalMaso}.

Now we introduce the following proposition in \cite{DalMaso}.

\begin{Proposition}[\cite{DalMaso} Proposition 5.4]\label{claim}
If $(F_h)$ is an increasing sequence, then
$$
\Gamma-\lim\limits_{h\rightarrow+\infty}F_h=\lim\limits_{h\rightarrow+\infty}sc^-F_h=\sup\limits_{h\in\mathbb{N}}sc^-F_h.
$$
\end{Proposition}

\begin{Theorem}\label{gammabelow}
For every sequence $\{s_j\}_j\subset(0,1)$ strictly increasing to $s\in(0,1)$, $1<p<+\infty$, let $\Omega$ be an open bounded set in $\mathbb{R}^N$, then it holds
$$
\Gamma-\lim_{j\rightarrow+\infty}\underline{\mathcal{E}}^{s_j}
=\Gamma-\lim_{j\rightarrow+\infty}\mathcal{E}^{s_j}=\underline{\mathcal{E}}^{s}.
$$
\end{Theorem}

\begin{proof}
Let $R$ denote the diameter of $\Omega$.
Define $F^s$ and $F^{s_j}$ as mapping $L^p(\Omega)$ to $[0,+\infty]$ by
$$
F^s(u)=R^{sp}\mathcal{E}^s(u),\ \underline{F}^{s}(u)=R^{sp}\underline{\mathcal{E}}^s(u).
$$
Then clearly $F^s$ and $\underline{F}^s$ are lower semi-continuous, and the sequences $\{F^{s_j}\}$ and $\{\underline{F}^{s_j}\}$ are both increasing and pointwise convergent to $\underline{F}^s$. Indeed, for $0<k\leq s<1$, this is just a simple calculation as
$$
\begin{array}{lll}
\int_{\Omega\times\Omega}\frac{|u(x)-u(y)|^p}{|x-y|^{N+kp}}dxdy&\leq&
\int_{\Omega\times\Omega}\frac{|u(x)-u(y)|^p}{|x-y|^{N+sp+(k-s)p}}dxdy\\
&\leq& R^{(s-k)p}\int_{\Omega\times\Omega}\frac{|u(x)-u(y)|^p}{|x-y|^{N+sp}}dxdy.
\end{array}
$$
Then thanks to Proposition \ref{claim}, we infer that
$$
\Gamma-\lim_{j\rightarrow+\infty}\underline{F}^{s_j}
=\Gamma-\lim_{j\rightarrow+\infty}F^{s_j}=\underline{F}^{s},
$$
and then the assertion easily follows.
\end{proof}

From the Theorem above, we can see that the result is not as smooth as the case in Theorem \ref{sconvaboveThm} to get the accumulation function belong to the ideal space $W^{s,p}_0(\Omega)$, but a wider space $W^{s^-,p}_0(\Omega)$. And as a byproduct we immediately establish the following results.
\begin{Corollary}\label{equiformCol}
For every $s\in(0,1)$, $1<p<+\infty$, let $\Omega$ be an open bounded set in $\mathbb{R}^N$, then the following conditions are equivalent:\\
\indent(i) For every sequence $\{s_j\}_j\subset(0,1)$ strictly increasing to $s\in(0,1)$, it holds
$$
\Gamma-\lim_{j\rightarrow+\infty}\mathcal{E}^{s_j}=\mathcal{E}^s;
$$
\indent(ii) $W^{s^-,p}_0(\Omega)=W^{s,p}_0(\Omega)$.
\end{Corollary}

\begin{Remark}
We want to mention that we can also establish similar result like Corollary \ref{equiformCol} in our $relative$-$nonlocal$ setting. For other equivalent forms one can refer to \cite{Li}, in which, we have also established some other equivalent forms of the space $\widetilde{W}^{s^-,p}_{0,tR}(\Omega)$ (see \cite{Li}) in the $relative$-$nonlocal$ setting under no regularity assumptions on $\partial\Omega$.
\end{Remark}

\medskip
{\bf Acknowledgements:}
Many thanks to Professor SUGIMOTO Mitsuru for his detailed discussion on the understanding of some concepts and some proof in this paper. And the author is also indebted with Professor Terasawa Yutaka for his helpful suggestions and comments on some results in this paper. The author was supported by CSC Scholarship of Chinese Government (Grant No. 201506190121).
\medskip


\end{document}